\documentclass[11pt]{article}
\usepackage{epsfig,mst-stylefile,amssymb,amsmath,harvard,color}
\usepackage[normalem]{ulem}
\usepackage[svgnames]{xcolor}

\newcommand{\bbZ}{{\Bbb Z}}
\newcommand{\bbR}{{\Bbb R}}
\newcommand{\bbN}{{\Bbb N}}
\newcommand{\bbC}{{\Bbb C}}

\newcommand{\bbQ}{{\Bbb Q}}

\newcommand{\norm}[1]{\left\|#1\right\|}

\renewcommand{\cite}{\citeyear}

\newtheorem{lem}[theorem]{Lemma}
\newtheorem{cor}[theorem]{Corollary}

\newtheorem{prop}[theorem]{Proposition}

\begin{document}

\title{Exponents of operator self-similar random fields
\thanks{The first author was supported in part by the ARO grant W911NF-14-1-0475.
The second author was supported in part by the NSF grants DMS-1462156 and EAR-1344280 and the ARO grant W911NF-15-1-0562. The third author was supported in part by the NSA grant H98230-13-1-0220. The authors would like to thank Hans-Peter Scheffler for the proof of Lemma \ref{e:A^(-1)w_k->0}.}
 }

\author{Gustavo Didier \\ Tulane University \and Mark M.\ Meerschaert \\ Michigan State University \and Vladas Pipiras \\ University of North Carolina}

\bibliographystyle{agsm}

\maketitle

\begin{abstract}
If $X(c^E t)$ and $c^H X(t)$ have the same finite-dimensional distributions for some linear operators $E$ and $H$, we say that the random vector field $X(t)$ is operator self-similar.  The exponents $E$ and $H$ are not unique in general, due to symmetry.  This paper characterizes the possible set of range exponents $H$ for a given domain exponent, and conversely, the set of domain exponents $E$ for a given range exponent.
\end{abstract}


\section{Introduction}
\label{s:intro}

A random vector is called {\it full} if its distribution is not supported on a lower dimensional hyperplane.  A random field $X = \{X(t)\}_{t \in \bbR^m}$ with values in $\bbR^n$ is called {\it proper} if $X(t)$ is full for all $t\neq 0$.  A linear operator $P$ on $\bbR^m$ is called a projection if $P^2=P$.  Any nontrivial projection $P\neq I$ maps $\bbR^m$ onto a lower dimensional subspace.  We say that a random vector field $X$ is {\it degenerate} if there exists a nontrivial projection $P$ such that $X(t)=X(Pt)$ for all $t\in\bbR^m$.  We say that $X$ is {\it stochastically continuous} if $X(t_n)\to X(t)$ in probability whenever $t_n\to t$.  A proper, nondegenerate, and stochastically continuous random vector field $X$ is called {\it operator self-similar} (o.s.s, or $(E,H)$-o.s.s.) if
\begin{equation}\label{e:o.s.s.}
\{X(c^E t)\}_{t \in \bbR^m} \simeq \{c^{H}X(t)\}_{t \in \bbR^m}\quad \text{for all }c> 0.
\end{equation}
In \eqref{e:o.s.s.}, $\simeq$ indicates equality of finite-dimensional distributions, $E \in M(m,\bbR)$ and $H \in M(n,\bbR)$, where $M(p,\bbR)$ represents the space of real-valued $p \times p$ matrices, and $c^{M} = \exp(M (\log c)) = \sum^{\infty}_{k=0} (M \log c)^k/ k!$ for a square matrix $M$. We will assume throughout this paper that the eigenvalues of $E$ and $H$ have (strictly) positive real parts. This ensures that $c^Et$ and $c^Hx$ tend to zero as $c\to 0$, and tend to infinity in norm as $c\to\infty$ for any $t,x\neq 0$, see Theorem 2.2.4 in Meerschaert and Scheffler \cite{meerschaert:scheffler:2001}.  Then it follows from stochastic continuity that $X(0)=0$ a.s. At the end of Section 2, we will discuss what happens if some eigenvalues of $H$ have zero real part.

Operator self-similar random (vector) fields are useful to model long-range dependent, spatial and spatio-temporal anisotropic data in hydrology, radiology, image processing, painting and texture analysis (see, for example, Harba et al.\ \cite{harba:jacquet:jennane:lousoot:benhamou:lespesailles:tourliere:1994}, Bonami and Estrade \cite{bonami:estrade:2003}, Ponson et al.\ \cite{ponson:bonamy:auradou:mourot:morel:bouchaud:guillot:hulin:2006}, Roux et al.\ \cite{roux:clausel:vedel:jaffard:abry:2013}).  For a stochastic process (with $m=n=1$), the relation \eqref{e:o.s.s.} is called self-similarity (see, for example, Embrechts and Maejima \cite{EmbrechtsMaejima}, Taqqu \cite{taqqu:2003}). Fractional Brownian motion is the canonical example of a univariate self-similar process, and there are well-established connections between self-similarity and the long-range dependence property of time series (see Samorodnitsky and Taqqu \cite{samorodnitsky:taqqu:1994}, Doukhan et al.\ \cite{doukhan:2003}, Pipiras and Taqqu \cite{pipiras:taqqu:2016}).

The theory of operator self-similar stochastic processes (namely, $m = 1$) was developed by Laha and Rohatgi \cite{laha:rohatgi:1981} and Hudson and Mason \cite{hudson:mason:1982}, see also Chapter 11 in Meerschaert and Scheffler \cite{meerschaert:scheffler:2001}. Operator fractional Brownian motion was studied by Didier and Pipiras \cite{didier:pipiras:2011,didier:pipiras:2012} (see also Robinson \cite{robinson:2008}, Kechagias and Pipiras \cite{kechagias:pipiras:2015:def,kechagias:pipiras:2016:ident} on the related subject of multivariate long range dependent time series). For scalar fields (with $n=1$), the analogues of fractional Brownian motion and fractional stable motion were studied in depth by Bierm\'e et al.\ \cite{bierme:meerschaert:scheffler:2007}, with related work and applications found in Benson et al.\ \cite{fractint}, Bonami and Estrade \cite{bonami:estrade:2003}, Bierm\'{e} and Lacaux \cite{bierme:lacaux:2009}, Bierm\'{e}, Benhamou and Richard \cite{bierme:benhamou:richard:2009}, Clausel and Vedel \cite{clausel:vedel:2011,clausel:vedel:2013}, Meerschaert et al.\ \cite{WRCR:WRCR20376}, and Dogan et al.\ \cite{GRL:GRL52237}.  Li and Xiao \cite{li:xiao:2011} proved important results on operator self-similar random vector fields, see Theorem \ref{t:Li&Xiao_theo2.2} below. Baek et al.\ \cite{baek:didier:pipiras:2014} derived integral representations for Gaussian o.s.s.\ random fields with stationary increments.

Domain exponents $E$ and range exponents $H$ satisfying \eqref{e:o.s.s.} are not unique in general, due to symmetry. More specifically, the set of domain or range exponents comprises more than one element if and only if the respective set of domain or range symmetries contains a vicinity of the identity. This paper describes the set of possible range exponents $H$ for a given domain exponent $E$, and conversely, the set of possible domain exponents $E$ for a given range exponent $H$.  In both cases, the difference between two exponents lies in the tangent space of the symmetries.  The corresponding result for o.s.s.\ stochastic processes, the case $m=1$, was established by Hudson and Mason \cite{hudson:mason:1982}. In the characterization of the sets of domain or range exponents, the key assumption is that of the existence of a range or a domain exponent, respectively. This allows us to make use of the framework laid out by Hudson and Mason \cite{hudson:mason:1982}, Li and Xiao \cite{li:xiao:2011} as well as that of Meerschaert and Scheffler \cite{meerschaert:scheffler:2001}, Chapter 5, the latter being more often used for establishing results for domain exponents. In addition, we provide a counterexample showing that the existence of one of the two exponents is a necessary condition for establishing the relation \eqref{e:o.s.s.}.

\section{Results}
\label{s:main}

This section contains the main results in the paper. All proofs can be found in Section \ref{s:proofs}.

The domain and range symmetries of $X$ are defined by
\begin{equation}\begin{split}\label{e:def_G1}
G^{\textnormal{dom}}_{1} :=& \{A \in M(m,\bbR):X(At)\simeq X(t)\},\\
G^{\textnormal{ran}}_1 :=& \{ B \in M(n,\bbR): BX(t) \simeq X(t)\}.
\end{split}\end{equation}

For the next proposition, let $GL(k,\bbR)$ be the general linear group on $\bbR^k$.
\begin{prop}\label{p:Gdom_Gran_are_compact_groups}
Let  $X = \{X(t)\}_{t \in \bbR^m}$ be a proper nondegenerate random field with values in $\bbR^n$ such that $X(0)=0$ a.s. Then, $G^{\textnormal{ran}}_{1}$ is a compact subgroup of $GL(n,\bbR)$, and  $G^{\textnormal{dom}}_{1}$ is a compact subgroup of $GL(m,\bbR)$.
\end{prop}

The definition \eqref{e:o.s.s.} is more general than it appears.  Given $E \in M(m,\bbR)$, a proper nondegenerate random field $X$ will be called \textit{$E$-range operator self-similar} if there exist invertible linear operators $B(c)\in M(n,\bbR)$ such that
\begin{equation}\label{e:r.o.s.s.}
\{X(c^E t)\}_{t \in \bbR^m} \simeq \{B(c)X(t)\}_{t \in \bbR^m}\quad \text{for all $c>0$}.
\end{equation}

\begin{theorem}\label{t:Li&Xiao_theo2.2}(Li and Xiao \cite{li:xiao:2011}, Theorem 2.2)
For any $E$-range operator self-similar random vector field $X$, there exists a linear operator $H \in M(n,\bbR)$
such that \eqref{e:o.s.s.} holds.
\end{theorem}

Given $H \in M(n,\bbR)$, we say that a proper nondegenerate random vector field $X$ is \textit{$H$-domain operator self-similar} if there exists an invertible linear operator $A(c)\in M(m,\bbR)$ such that
\begin{equation}\label{e:d.o.s.s.}
\{X(A(c) t)\}_{t \in \bbR^m} \simeq \{c^{H}X(t)\}_{t \in \bbR^m}\quad \text{for all }c > 0.
\end{equation}

\begin{theorem}\label{t:existence_of_domain_exponents}
For any $H$-domain operator self-similar random vector field $X$, there exists a linear operator $E \in M(m,\bbR)$
such that \eqref{e:o.s.s.} holds.
\end{theorem}

\begin{remark}\label{r:cyclo}
One could also consider a more general scaling relation $X(A(c)t) \simeq B(c)X(t)$, but this need not lead to an o.s.s.\ field even in the case $m=n=1$.  For example, let $b,c_0 > 1$ be constants such that $\alpha := \log c_0/\log b \in (0,1)$, and let $\phi(dy)$ be a discrete L\'{e}vy measure defined by $\phi(\{b^k\}) = c^{-k}_0$, $k \in \bbZ$. Now define a probability measure $\nu$ by means of its characteristic function $\exp \psi(\theta)$, where
$$
\psi(\theta) = \int_{\bbR} (e^{i \theta y}-1)\phi(dy) = \sum^{\infty}_{k = - \infty}(e^{i \theta b^k}-1)c^{-k}_0.
$$
Then, $\psi(b \theta) = c_0 \psi(\theta)$, and thus $\nu^{c_0} = b \nu = c^{1/\alpha}_0\nu$ (here, $b \nu(dx) := \nu(b^{-1}dx)$, so that if $\nu$ is the probability measure of a random variable $Y$, then $b \nu$ is the probability measure of the random variable $bY$). Then, $\nu$ is a strictly $(b,c_0)$ semistable distribution with $\alpha=\log c_0/\log b$. If $\{X(t)\}_{t \in \bbR}$ is a L\'evy process such that $X(1)$ has distribution $\nu$, it follows that $X(c_0t) \simeq c_0^{1/\alpha}X(t)$, i.e., $X$ is semi-self-similar (see Maejima and Sato \cite{Maejimasss}). Taking $A(c)=c_0$ and $B(c)=c_0^{1/\alpha}$ yields a process with the general scaling, but since the f.d.d.\ equality only holds for $c=c_0^k$, $k \in \bbN$, the process is not o.s.s.
\end{remark}

Given an o.s.s.\ random field $X$ with domain exponent $E$, the set of all possible range exponents $H$ in \eqref{e:o.s.s.} will be denoted by ${\mathcal E}^{\textnormal{ran}}_{E}(X)$. Given a range exponent $H$, we denote by ${\mathcal E}^{\textnormal{dom}}_{H}(X)$ the set of all possible domain exponents.  Given a closed group ${\mathcal G} \subseteq GL(m,\bbR)$, one can define its tangent space
\begin{equation}\label{e:T(G)}
T({\mathcal G}) = \Big\{A \in M(n,\bbR): A = \lim_{n \rightarrow \infty} \frac{G_n - I}{d_n}, \quad
\textnormal{for some } \{G_n\} \subseteq {\mathcal G} \textnormal{ and some }
0 \neq d_n \rightarrow 0 \Big\}.
\end{equation}
The next two theorems are the main results of this paper.

\begin{theorem}\label{t:E=H+TG_range}
Given an o.s.s.\ random vector field $X$ with domain exponent $E$, for any range exponent $H$ we have
\begin{equation}\label{e:E=H+TG_range}
{\mathcal E}^{\textnormal{ran}}_E (X) = H + T(G^{\textnormal{ran}}_1).
\end{equation}
Moreover, we can always choose an exponent $H_0 \in {\mathcal E}^{\textnormal{ran}}_E(X)$ such that
\begin{equation}\label{e:B0_A=AB_0_X_range}
H_0 A = A H_0 \quad \text{for every }A \in G^{\textnormal{ran}}_1.
\end{equation}
The nilpotent part of every $H \in {\mathcal E}^{\textnormal{ran}}_E(X)$ is the same, and it commutes with every element of the symmetry group $G^{\textnormal{ran}}_{1}$. Furthermore, every matrix $H \in {\mathcal E}^{\textnormal{ran}}_E(X)$ has the same real spectrum (real parts of the eigenvalues).
\end{theorem}

\begin{theorem}\label{t:domain_Exponents}
Given an o.s.s.\ random vector field $X$ with range exponent $H$, for any domain exponent $E$ we have
\begin{equation}\label{e:E(varphi)=B+T(S(varphi))}
{\mathcal E}^{\textnormal{dom}}_H(X) = E + T(G^{\textnormal{dom}}_1).
\end{equation}
Moreover, we can always choose an exponent $E_0 \in {\mathcal E}^{\textnormal{dom}}_H(X)$
such that
\begin{equation}\label{e:B0_A=AB_0_X}
E_0 B = B E_0, \quad \text{for all }B\in G^{\textnormal{dom}}_1.
\end{equation}
The nilpotent part of every $E \in {\mathcal E}^{\textnormal{dom}}_H(X)$ is the same, and it commutes with every element of the symmetry group $G^{\textnormal{dom}}_{1}$. Furthermore, every matrix $E \in {\mathcal E}^{\textnormal{dom}}_H(X)$ has the same real spectrum.
\end{theorem}

In the next example, and throughout the paper, $O(k)$ denotes the orthogonal group in $GL(k,\bbR)$.
\begin{example}
Let $X = \{X(t)\}_{t \in \bbR^2}$ be an $\bbR^2$-valued operator fractional Brownian field (OFBF), namely, a zero mean Gaussian, o.s.s., stationary increment random field with covariance function ${\Bbb E}X(s)X(t)^* = \Gamma(s,t)$, $s,t \in \bbR$. From the Gaussian assumption,
$$
G^{\textnormal{dom}}_1 = \{A \in M(2,\bbR): \Gamma(As,At) = \Gamma(s,t), \hspace{1mm}s, t \in \bbR\},
$$
$$
G^{\textnormal{ran}}_1 = \{B \in M(2,\bbR): B\Gamma(s,t)B^* = \Gamma(s,t),\hspace{1mm} s, t \in \bbR\}.
$$
In addition, assume that $X$ has a spectral density $f_X(x) = \|x\|^{-\gamma} I $, $x \in \bbR^2 \backslash \{0\}$, $2 < \gamma < 4$, where $\|\cdot\|$ denotes the Euclidean norm and $I$ is the identity matrix. This means that its covariance function can be written as
\begin{equation}\label{e:OFBF}
\Gamma(s,t) = I \int_{\bbR^2} (e^{i \langle s,x\rangle}-1)(e^{-i \langle t,x\rangle}-1)\frac{1}{\|x\|^{\gamma}}\hspace{1mm} dx,
\end{equation}
where $\langle \cdot, \cdot \rangle$ is the Euclidean inner product. By \eqref{e:OFBF} and a change of variables, $X$ is $(E,H)$-o.s.s.\ with $E = I$, $H = h I$, where $h = (\gamma-2)/2$. It is clear that $H$ and $E$ are commuting exponents (see \eqref{e:B0_A=AB_0_X_range} and \eqref{e:B0_A=AB_0_X}). Since $\Gamma(s,t)$ is a scalar matrix for $s,t \in \bbR^2$, then the condition
\begin{equation}\label{e:AGamma(s,t)A*=Gamma(s,t)}
A\Gamma(s,t)A^* = \Gamma(s,t)
\end{equation}
for $A \in GL(2,\bbR)$ implies that $AA^* = I$, namely, $A \in O(2)$. Moreover, any $A \in O(2)$ satisfies \eqref{e:AGamma(s,t)A*=Gamma(s,t)}. Hence, $G^{\textnormal{ran}}_1 = O(2)$. Now note that, by a change of variables in \eqref{e:OFBF} and the continuity of the spectral density except at zero, $A \in G^{\textnormal{dom}}_1 \Leftrightarrow \|A^* x\| = \|x\|$, $x \in \bbR^{m}\backslash\{0\}$, i.e., $A \in O(2)$. As a consequence, $G^{\textnormal{dom}}_1 = O(2)$. Therefore, from \eqref{e:E=H+TG_range} and \eqref{e:E(varphi)=B+T(S(varphi))},
$$
{\mathcal E}^{\textnormal{ran}}_I (X) = h I + so(2) , \quad {\mathcal E}^{\textnormal{dom}}_{hI}(X) = I + so(2),
$$
where $so(2) = T(O(2)) \subseteq M(2,\bbR)$ is the space of $2 \times 2$ skew-symmetric matrices.
\end{example}

\begin{cor}\label{c:relation_between_domain_and_range_exponents} Given an o.s.s.\ random vector field $X$:
\begin{enumerate}
\item [($a$)] If $E_1$, $E_2$ are two domain exponents for $X$, then for any $H_1\in {\mathcal E}^{\textnormal{ran}}_{E_1}(X)$ and $H_2\in {\mathcal E}^{\textnormal{ran}}_{E_2}(X)$ we have
\begin{equation}\label{e:relation_between_range_exponents}
{\mathcal E}^{\textnormal{ran}}_{E_1}(X) - H_{1} = {\mathcal E}^{\textnormal{ran}}_{E_2}(X) - H_{2}.
\end{equation}
\item [($b$)] If $H_1$, $H_2$ are two range exponents for $X$, then for any $E_1\in {\mathcal E}^{\textnormal{dom}}_{H_1}(X)$ and $E_2\in {\mathcal E}^{\textnormal{dom}}_{H_2}(X)$ we have
\begin{equation}\label{e:relation_between_domain_exponents}
{\mathcal E}^{\textnormal{dom}}_{H_1}(X) - E_{1} = {\mathcal E}^{\textnormal{dom}}_{H_2}(X) - E_{2}.
\end{equation}
\end{enumerate}
\end{cor}

Hudson and Mason \cite{hudson:mason:1982} also considered o.s.s.\ stochastic processes for which the eigenvalues of the range exponent $H$ can have zero real parts.  In this case, the process can be decomposed into two component processes of lower dimension. One is associated with the eigenvalues of $H$ with null real parts, and the resulting random field has constant sample paths; the other has a range exponent whose eigenvalues all have positive real parts, and equals zero at $t=0$ a.s.  Next we show that the same is true for random fields.  Hence the condition assumed throughout the rest of this paper, that every eigenvalue of $H$ has positive real part, entails no significant loss of generality.

\begin{theorem}\label{t:X=X1+X2}
Let $X$ be a proper, stochastically continuous random vector field that satisfies the scaling relation \eqref{e:o.s.s.} for some $E$ whose eigenvalues all have positive real part.  Then, there exists a direct sum decomposition  $\bbR^n = V_1 \oplus V_2$ into $H$-invariant subspaces such that, writing $X=X_1+X_2$ and $H=H_1\oplus H_2$ with respect to this decomposition:
\begin{itemize}
\item [(i)] $X_1$ has constant sample paths; and
\item [(ii)] $X_2$ is $(E,H_2)$-o.s.s.\ with $X_2(0) = 0$ a.s.
\end{itemize}
\end{theorem}


\section{Proofs}\label{s:proofs}

Lemmas \ref{e:A^(-1)w_k->0}--\ref{l:G1_is_not_a_neighborhood_of_I}, to be stated and defined next, will be used in the proofs of Proposition \ref{p:Gdom_Gran_are_compact_groups} and Theorem \ref{t:existence_of_domain_exponents}. Define the operator norm $\| A \|= \sup \{ \|Aw\|: \| w \|=1 \}$ for any $A\in M(m,\bbR)$. Hereinafter the symbol $X \stackrel{d}= Y$ denotes the equality in distribution of two random vectors or variables $X$ and $Y$.

\begin{lem}\label{e:A^(-1)w_k->0}
Let $\{A_k\}_{k \in \bbN} \subseteq GL(m,\bbR)$ such that $\| A_k \|\rightarrow \infty$. Then, there exists a sequence $\{w_{k}\} \subseteq S^{m-1}_{\bbR}:= \{w\in\bbR^m:\|w\|=1\}$  such that $A^{-1}_{k} w_{k}\rightarrow 0$.
\end{lem}
\begin{proof}
By compactness and continuity, there exists a sequence $\{v_k\}_{k \in \bbN} \in S^{m-1}_{\bbR}$ such that $\| A_k v_k  \| \rightarrow \infty$. Now let
$\frac{A_k v_k}{\| A_k v_k  \|} \in S^{m-1}_{\bbR}$. Then,
$$
\| A^{-1}_k w_k \| = \frac{1}{\| A_k v_k  \| } \rightarrow 0, \quad k \rightarrow \infty,
$$
which establishes the claim. $\Box$\\
\end{proof}

{\sc Proof of Proposition \ref{p:Gdom_Gran_are_compact_groups}}: Since $X$ is proper and nondegenerate, it is easy to check that $G^{\textnormal{dom}}_1$ and $G^{\textnormal{ran}}_1$ are groups.  Hence we need only establish their topological properties. We first look at $G^{\textnormal{ran}}_1$. To show closedness (in the relative topology of $GL(n,\bbR)$, c.f.\ Lemma \ref{l:zeta_is_cont}), let $G^{\textnormal{ran}}_1 \ni A_{k} \rightarrow A \in GL(n,\bbR)$. Then,
$$
(X(t_1), \hdots, X(t_j) ) \stackrel{d}= (A_k X(t_1), \hdots, A_k X(t_j) ) \stackrel{P}\rightarrow (A X(t_1), \hdots, A X(t_j) ), \quad k \rightarrow \infty,
$$
i.e., $A \in G^{\textnormal{ran}}_1$. As for boundedness, by contradiction assume that there exists some $\{A_{k} \}_{k \in \bbN} \subseteq G^{\textnormal{ran}}_1$ such that $\|A_k\| \rightarrow \infty$. Then, $X(t_0) \stackrel{d}= A^{-1}_k X(t_0)$, $t_0 \neq 0$. Since we also have $\|A^*_k\| \rightarrow \infty$, then by the auxiliary Lemma \ref{e:A^(-1)w_k->0} there is a convergent subsequence $\{w_{k'}\} \subseteq S^{n-1}_{\bbR}$, $w_{k'}\rightarrow w_0$, such that $(A^*_{k'})^{-1}w_{k'} \rightarrow 0$, $k' \rightarrow \infty$. Consequently,
$$
w^*_{0}X(t_0) \leftarrow w^*_{k'}X(t_0) \stackrel{d}= w^*_{k'}A^{-1}_{k'} X(t_0) \stackrel{P}\rightarrow 0.
$$
This contradicts the properness of $X(t_0)$.

We now turn to $G^{\textnormal{dom}}_1$. To show closedness, take $\{A_k\}_{k \in \bbN} \subseteq G^{\textnormal{dom}}_{1}$ such that $A_k \rightarrow A \in M(m,\bbR)$. Consider any $j$-tuple $t_1,\hdots,t_j \in \bbR^m$. Then,
$$
(X(t_1),\hdots,X(t_j) ) \stackrel{d}= ( X(A_k  t_1),\hdots, X(A_k  t_j) ) \stackrel{P}\rightarrow (X(A t_1),\hdots, X(A t_j) ),
$$
where convergence follows from stochastic continuity. Then, $A \in G^{\textnormal{dom}}_{1}$. Since $X$ is nondegenerate, then $A \in GL(m,\bbR)$. Thus, $G^{\textnormal{dom}}_{1}$ is closed in the latter group.

To show boundedness, by contradiction suppose that there exists $\{A_k\}_{k \in \bbN} \subseteq G^{\textnormal{dom}}_{1}$ such that $\norm{A_k} \rightarrow \infty$. By Lemma \ref{e:A^(-1)w_k->0}, there is a subsequence $\{w_{k'}\} \subseteq S^{m-1}_{\bbR}$, $w_{k'} \rightarrow w_0$, such that $A^{-1}_{k}w_{k'} \rightarrow 0$. Therefore, since $X(0)=0$ a.s.,
$$
0 = X(0) \stackrel{P}\leftarrow X(A^{-1}_{k'} w_{k'}) \stackrel{d}= X(w_{k'}) \stackrel{P}\rightarrow X(w_0).
$$
This contradicts the properness of $X(w_0)$. $\Box$\\

The next lemmas show that an $H$-domain o.s.s.\ random vector field $X$ must satisfy a domain scaling law. For any $\lambda > 0$ and any $C_{\lambda} \in GL(m,\bbR)$ such that
\begin{equation}\label{e:def_Clambda_X}
X(C^{-1}_{\lambda}t) \simeq \lambda^{H}X(t),
\end{equation}
let $G_\lambda$ denote the class of matrices defined by
\begin{equation}\label{e:def_Glambda_X}
G_{\lambda} = C_{\lambda} G^{\textnormal{dom}}_1 \neq \emptyset.
\end{equation}
Note that, since $X$ is domain o.s.s., the set $G_\lambda$ is not empty. Also, note that $G^{\textnormal{dom}}_1=G_1$.

\begin{lem}\label{l:Glambda_welldefined_1}
A matrix $D \in GL(m,\bbR)$ satisfies
\begin{equation}\label{e:varphi(C^(-1)x)=lambda^(-alphaQ)varphi(x)}
X(D^{-1} t) \simeq \lambda^{H}X(t)
\end{equation}
if and only if
\begin{equation}\label{e:C_in_Glambda}
D \in G_{\lambda}.
\end{equation}
\end{lem}

\begin{proof}
Assume (\ref{e:varphi(C^(-1)x)=lambda^(-alphaQ)varphi(x)}) holds. Then, $X(D^{-1}C_{\lambda}t)\simeq\lambda^H X(C_{\lambda}t) \simeq X(C^{-1}_{\lambda}C_{\lambda}t) = X(t)$ by \eqref{e:def_Clambda_X}.
Therefore, $C^{-1}_{\lambda} D \in
G^{\textnormal{dom}}_{1}$, whence $D \in C_{\lambda} G^{\textnormal{dom}}_{1} = G_{\lambda}$. Conversely, assume (\ref{e:C_in_Glambda}) holds. Then, there exists $S \in G^{\textnormal{dom}}_{1}$ such that $D=C_\lambda S$, and so $X(D^{-1} t) \simeq X(S^{-1}C^{-1}_{\lambda} t)
\simeq X(C^{-1}_{\lambda} t) \simeq \lambda^{H}X(t)$.
$\Box$\\
\end{proof}

\begin{lem}\label{l:Glambda=CS(varphi)}
For any matrix $C \in G_{\lambda}$ we can write
\begin{equation}\label{e:Glambda=C_G1}
G_{\lambda} = C G^{\textnormal{dom}}_{1}.
\end{equation}
Moreover, for any choice of $C_\lambda, D_\lambda \in GL(m,\bbR)$ satisfying condition \eqref{e:def_Clambda_X}, $G_{\lambda} = C_{\lambda} G^{\textnormal{dom}}_1 =D_{\lambda} G^{\textnormal{dom}}_1$.
\end{lem}
\begin{proof}
Let $C \in G_{\lambda}$. If $D = CS$ for some $S \in G^{\textnormal{dom}}_{1}$, then by Lemma \ref{l:Glambda_welldefined_1},
$$
X(D^{-1}t) \simeq X(S^{-1}C^{-1}t) \simeq X(C^{-1}t) \simeq
\lambda^{H}X(t).
$$
Thus, again by Lemma \ref{l:Glambda_welldefined_1}, $D \in G_{\lambda}$. This shows that $CG^{\textnormal{dom}}_{1}\subset G_\lambda$.

Conversely, if $D \in G_{\lambda}$, then for any $C\in G_{\lambda}$ we have $X(D^{-1}t) \simeq \lambda^H X(t) \simeq X(C^{-1}t)$, by Lemma \ref{l:Glambda_welldefined_1}, which implies that
$X(C^{-1}Dt) \simeq X(D^{-1}Dt) = X(t)$. Thus, $C^{-1}D \in G^{\textnormal{dom}}_{1}$. Therefore, $D \in C G^{\textnormal{dom}}_{1}$, proving $G_\lambda \subset CG^{\textnormal{dom}}_{1}$ and establishing \eqref{e:Glambda=C_G1}.

Now let $C_{\lambda}, D_{\lambda} \in GL(m,\bbR)$ be two matrices satisfying \eqref{e:def_Clambda_X}. Lemma \ref{l:Glambda_welldefined_1} implies that $C_{\lambda}, D_{\lambda} \in G_{\lambda}$. Thus, \eqref{e:Glambda=C_G1} implies that $G_{\lambda} = C_{\lambda} G^{\textnormal{dom}}_1 = D_{\lambda} G^{\textnormal{dom}}_1$. $\Box$\\
\end{proof}

\begin{lem}\label{l:Glambda_welldefined_3}
For any $\lambda, \mu > 0$, if $G_{\lambda} \cap
G_{\mu} \neq \emptyset$, then $G_{\lambda} = G_{\mu}$.
\end{lem}
\begin{proof}
Assume there exists $A \in GL(m,\bbR)$ such that $A \in  G_{\lambda} \cap G_{\mu} = C_{\lambda} G^{\textnormal{dom}}_1 \cap C_{\mu} G^{\textnormal{dom}}_1$ for some $C_{\lambda}, C_{\mu} \in GL(m,\bbR)$. Therefore, there exist $S_{\lambda}, S_{\mu} \in G^{\textnormal{dom}}_1$ such that
$A = C_{\lambda} S_{\lambda} = C_{\mu} S_{\mu}$. Thus, $C_{\mu} = C_{\lambda}(S_{\lambda}S^{-1}_{\mu})$. Consequently, for all $A_{\mu} \in G_{\mu}$, there exists
$S_{A_{\mu}} \in G^{\textnormal{dom}}_1$ such that
$$
A_{\mu} = C_{\mu}S_{A_{\mu}} =
C_{\lambda}(S_{\lambda}S^{-1}_{\mu}S_{A_{\mu}}),
$$
where $S_{\lambda}S^{-1}_{\mu}S_{A_{\mu}} \in G^{\textnormal{dom}}_1$. Then $A_{\mu} \in G_{\lambda}$, which shows that $G_{\mu}\subseteq G_{\lambda}$. The same argument can be used for the converse. $\Box$
\end{proof}
\begin{lem}\label{l:G_is_group}
For matrix classes $G_{\lambda}$, $G_{\mu}$, $\lambda, \mu > 0$, as in \eqref{e:def_Glambda_X}, define the product relation
\begin{equation}\label{e:Glambda*Gmu}
G_{\lambda}G_{\mu} = \{A \in M(m,\bbR): A = C_{\lambda} S_{\lambda}C_{\mu}S_{\mu}, \hspace{1mm}\textnormal{for some }S_{\lambda},S_{\mu} \in G^{\textnormal{dom}}_1\}, \quad \lambda, \mu > 0.
\end{equation}
Then, under \eqref{e:Glambda*Gmu}, the set
\begin{equation}\label{e:G}
G := \bigcup_{\lambda > 0} G_{\lambda}
\end{equation}
is a group of equivalence classes $G_{\bullet}$ of matrices in $GL(m,\bbR)$.
\end{lem}
\begin{proof}
Let $C \in G_{\lambda}$ and $\lambda > 0$. Since $X$ is nondegenerate, there exist $C_{\lambda}$ and $S_C \in G^{\textnormal{dom}}_1$ such that
$$
X(t) = X(C^{-1}Ct) = X(S^{-1}_{C}C^{-1}_{\lambda}Ct) \simeq
X(C^{-1}_{\lambda}Ct) \simeq  \lambda^{H}X(Ct).
$$
Thus, $C^{-1} \in G_{1/ \lambda}$, which implies that
$G^{-1}_{\lambda} \subseteq G_{1/ \lambda}$. By taking $\sigma =
1/\lambda$, this in turn implies that $G_{1/\sigma} \subseteq
G^{-1}_{\sigma}$, and therefore $G_{1/ \lambda} \subseteq G^{-1}_{\lambda}$.
As a consequence,
\begin{equation}\label{e:G^(-1)lambda=G_1/lambda}
G^{-1}_{\lambda} = G_{1/ \lambda} \in G.
\end{equation}
Now take $C \in G_{\lambda}, D \in G_{\mu}$. Then
$$
X(D^{-1}C^{-1}t) \simeq \mu^{H}X(C^{-1}t) \simeq
\mu^{H}\lambda^{H}X(t) = (\mu
\lambda)^{H}X(t).
$$
Thus, $CD \in G_{\mu \lambda}$ and, consequently, $G_{\lambda} G_{\mu} \subseteq G_{\lambda \mu}$. By taking $r = 1/\lambda$, $s = \lambda \mu$, we also obtain that $G_{1/r}G_{rs} \subseteq G_{s}$. Expression (\ref{e:G^(-1)lambda=G_1/lambda}) then implies that $G^{-1}_{r}G_{rs} \subseteq
G_{s}$. Thus, $G_{rs} \subseteq G_{r} G_{s}$. Therefore,
\begin{equation}\label{e:GmuGlambda=Gmulambda}
G_{\mu} G_{\lambda} = G_{\mu \lambda} \in G.
\end{equation}
Consequently, $G$ is a group, as claimed. $\Box$\\
\end{proof}

\begin{lem}\label{l:Gr_inter_Gs=empty}
For $\lambda \neq \mu$, $G_{\lambda} \cap G_{\mu} = \emptyset$.
\end{lem}
\begin{proof}
We argue by contradiction. By Lemma \ref{l:Glambda_welldefined_3}, if $G_{\lambda} \cap G_{\mu} \neq \emptyset$, then $G_\lambda = G_\mu$. Without loss of generality, assume that
$\lambda <\mu$. By \eqref{e:GmuGlambda=Gmulambda}, for any $t > 0$, $G_{t\mu} = G_t G_\mu = G_t G_\lambda = G_{t\lambda}$. By taking $t = 1/\mu$, we obtain $G_{\lambda/\mu} = G_{1}$. Thus, for any $t > 0$, $G_{t\lambda/\mu} = G_{t}$. By taking $t = \lambda^k/\mu^k$, we obtain a system of equalities leading to the conclusion that $G_{{\lambda^k}/{\mu^k}}=G_1$, $k \in \bbN$. Thus,
$$
X(t) = X(It) \simeq \Big(\frac{\lambda^k}{\mu^k}\Big)^{H}X(t).
$$
Since every eigenvalue of $H$ has positive real part, a straightforward computation using the Jordan decomposition of $H$ shows that $c^H x\to 0$ as $c\to 0$ for any $x\in\bbR^n$, see Theorem 2.2.4 in Meerschaert and Scheffler \cite{meerschaert:scheffler:2001}. It follows that
$\norm{(\lambda^k/\mu^k)^{H}X(t)}
\stackrel{P}\rightarrow 0$, $k \rightarrow \infty$. We arrive at a contradiction because $X$ is proper. $\Box$\\
\end{proof}

\begin{lem}\label{l:zeta_is_homomorph}
The mapping $\zeta: G \rightarrow \bbR_{+}$ defined by $\zeta(C) = \lambda$ when $C\in G_{\lambda}$ is a group homomorphism.
\end{lem}

\begin{proof}
Lemma \ref{l:Gr_inter_Gs=empty} shows that $\zeta$ is well-defined. Suppose that $C \in G_{\lambda}$, $D \in G_{\mu}$. Then Lemma \ref{l:G_is_group} shows that $CD \in G_{\lambda \mu}$, and $\zeta(CD) = \zeta(C) \zeta(D)$. $\Box$\\
\end{proof}

Given a topogical space $Z$, the subspace (or relative) topology on a subset $U\subseteq Z$ consists of all sets $O\cap U$ where $O$ is an open subset of $Z$.  For example, if $Z=\bbR$ and $U=[0,\infty)$, then $[0,1)$ is an open subset of $U$ in this topology.  For another example, the set $U=\{\lambda I:\lambda>0\}$ is not a closed subset of $M(m,\bbR)$. However, by considering sequences of matrices in $U$, it can be seen that the latter is a closed subset of $Z=GL(m,\bbR)$ in the relative topology, thus implying that $Z\setminus U$ is an open subset of $Z$.

\begin{lem}\label{l:zeta_is_cont}
The group $G$ in \eqref{e:G} is a closed subgroup of $GL(m,\bbR)$ in the relative topology, and $\zeta$ is a continuous function.
\end{lem}
\begin{proof}
Suppose that $\{D_{k}\}_{k \in \bbN} \subseteq G$ and that $D_k \rightarrow D$ in $GL(m,\bbR)$ as $k \rightarrow \infty$. Then, for all $k$, $D_k \in G_{\lambda_k}$ for some $\lambda_k > 0$. We first need to show that $D \in G$.

If for some subsequence $\{\lambda_{k^{'}}\}$, $\lambda_{k^{'}} \rightarrow \infty$ as $k'\rightarrow \infty$, then $\lambda^{-H}_{k^{'}} \rightarrow 0$ by Theorem 2.2.4 in Meerschaert and Scheffler \cite{meerschaert:scheffler:2001}, since every eigenvalue of $H$ has positive real part. Moreover, for $t \neq 0$, stochastic continuity yields $X(D^{-1}_{k'}t) \stackrel{d}\rightarrow X(D^{-1} t)$ as $k'\rightarrow \infty$. Therefore, in view of \eqref{e:def_Clambda_X} we have $ X(t)  \stackrel{d}=\lambda^{-H}_{k^{'}}X(D^{-1}_{k'}t) \stackrel{P}\rightarrow 0$, which contradicts properness. Therefore, $\{\lambda_{k}\}$ is relatively compact and there is a convergent subsequence such that $\lambda_{k^{'}} \rightarrow \lambda_0$ as $k'\rightarrow \infty$.
If $\lambda_0 = 0$, then for any $t_0 \neq 0$ we have $X(D^{-1}t_0) \stackrel{P}\leftarrow X(D^{-1}_{k'}t_0)\stackrel{d}= \lambda^{H}_{k'}X(t_0) \stackrel{P}\rightarrow 0$, which again contradicts properness. Therefore, $\lambda_0 > 0 $ and
$$
X(D^{-1}t)\stackrel{P}\leftarrow  X(D^{-1}_{k^{'}}t) \stackrel{d}= \lambda^{H}_{k^{'}} X(t) \stackrel{P}\rightarrow \lambda^{H}_0 X(t).
$$
Thus, $D \in G_{\lambda_0}$, and by Lemma \ref{l:G_is_group}, $G$ is a closed subgroup of $GL(m,\bbR)$ in the relative topology, as stated.

Let us turn back to the original sequence $\{\lambda_k\}_{k \in \bbN}$ of scalars associated with $\{D_k\}_{k \in \bbN}$. We claim that there does not exist a convergent subsequence $\{\lambda_{k^*}\} \subseteq \{\lambda_k\}$ such that $\lambda_{k^*} \rightarrow \mu \neq \lambda_0$. Otherwise, $D \in G_{\lambda_0} \cap G_{\mu} = \emptyset$ by Lemma \ref{l:Gr_inter_Gs=empty}, which is a contradiction. As a consequence, for any subsequence $\{\lambda_{k'}\}$, by relative compactness there exists a further subsequence $\{\lambda_{k''}\} \subseteq \{\lambda_{k'}\}$ such that $\lambda_{k''} \rightarrow \lambda_0$. This is equivalent to saying that $\zeta(D_k) = \lambda_{k} \rightarrow \lambda_0 = \zeta(D)$ as $k \rightarrow \infty$, i.e., the mapping $\zeta$ is continuous. $\Box$\\
\end{proof}

%

\begin{lem}\label{l:G1_is_not_a_neighborhood_of_I}
$G^{\textnormal{dom}}_1$ is not a neighborhood of $I$ in $G$.
\end{lem}
\begin{proof}
We need to build a sequence of matrices $D_k$ in $G \backslash G^{\textnormal{dom}}_1$ such that $D_k \rightarrow I$ as $k \rightarrow \infty$. So, let $r_k = 1 + \frac{1}{k}$, $k \in \bbN$. Since $G_{r} \neq \emptyset$, we can choose $C_k \in G_{r_k}$ such that
$$
X(C^{-1}_{k}t) \simeq r^{H}_{k}X(t), \quad k \in \bbN.
$$
Assume by contradiction that $\{C^{-1}_{k}\}$ is not relatively compact in the relative topology of $GL(m,\bbR)$. This means that $\norm{C^{-1}_{k}} \rightarrow \infty$, because by Lemma \ref{l:zeta_is_cont}, $G$ is a closed subgroup of $GL(m,\bbR)$ in the same topology. So, by Lemma \ref{e:A^(-1)w_k->0}, there exists a sequence $\{ t_{k'} \} \subseteq S^{m-1}_{\bbR}$ such that $C_{k'} t_{k'}\rightarrow 0$ and $t_{k'}\rightarrow t_0 \neq 0$. By \eqref{e:G^(-1)lambda=G_1/lambda}, stochastic continuity and the assumption that $X(0)=0$ a.s.,
$$
0 = X(0)\stackrel{d}\leftarrow X(C_{k'} t_{k'}) \stackrel{d}= (r^{-1}_{k'})^{H}X(t_{k'}) \stackrel{d}\rightarrow X(t_0), \quad k' \rightarrow \infty.
$$
This contradicts the properness of $X(t_0)$.

As a consequence, there is a subsequence $\{C^{-1}_{k'}\}$ such that $C^{-1}_{k'} \rightarrow C^{-1}$, where $C \in G$ by Lemma \ref{l:zeta_is_cont}. Stochastic continuity then yields
$$
X(t) \stackrel{d}\leftarrow r^{H}_{k'}X(t)\stackrel{d}= X(C^{-1}_{k'}  t) \stackrel{P}\rightarrow  X(C^{-1}t).
$$
This implies that $C \in G^{\textnormal{dom}}_1$. Define $D_{k'}= C^{-1}_{k'}C$. Then, $D_{k'}\rightarrow I$ and $D_{k'} \in G_{r^{-1}_{k'}}$. By Lemma \ref{l:Gr_inter_Gs=empty}, $D_{k'} \notin G^{\textnormal{dom}}_1$, $k' \in \bbN$, which concludes the proof. $\Box$
\end{proof}

\medskip

{\sc Proof of Theorem \ref{t:existence_of_domain_exponents}}: Although our result complements that of  Li and Xiao \cite{li:xiao:2011}, it builds upon domain-based (as opposed to range-based) concepts, and thus is closer in spirit to Meerschaert and Scheffler \cite{meerschaert:scheffler:2001}, Chapter 5. 
By Lemma \ref{l:zeta_is_cont}, $G (\ni I)$ is a subgroup of $GL(m,\bbR)$ which is closed in the relative topology of the latter. Then, the image of $T(G)$ under the exponential map $\exp(\cdot)$ (as defined by the matrix exponential) is a neighborhood of $I$ in $G$ (e.g., see Meerschaert and Scheffler \cite{meerschaert:scheffler:2001}, Proposition 2.2.10.d). By Lemma \ref{l:G1_is_not_a_neighborhood_of_I}, $G^{\textnormal{dom}}_1$ is not a neighborhood of $I$ in $G$; therefore, there exists $A \in T(G)$ such that $e^{A} \notin G^{\textnormal{dom}}_1$. Recall the function $\zeta$ from Lemma \ref{l:zeta_is_homomorph} and define the mapping $\bbR \ni s \mapsto f(s):=\log \zeta(e^{sA})$. Then, by Lemma \ref{l:zeta_is_homomorph},
$$
f(s+r) = \log \zeta(e^{(s+r)A}) = \log \zeta(e^{sA}e^{rA}) = \log \zeta(e^{sA}) + \log \zeta(e^{rA}) = f(s) + f(r),
$$
for $s , r \in \bbR$. Therefore, $f$ is a continuous additive homomorphism. Thus, there exists $\beta \in \bbR$ such that $f(s)=\beta s$ (e.g., see Hudson and Mason \cite{hudson:mason:1982}, p.\ 288). Moreover, if $\beta = 0$, then $e^{1A}= e^{A} \in G^{\textnormal{dom}}_1$, which is a contradiction. Thus, $\beta \neq 0$, and we can take $E:= \beta^{-1}A$ to obtain $\log \zeta(r^{E}) = \log r$ for $r > 0$. Therefore, $r^{E} \in G_r$ for $r > 0$. By Lemma \ref{l:Glambda_welldefined_1}, \eqref{e:o.s.s.} holds. $\Box$\\

{\sc Proof of Theorem \ref{t:E=H+TG_range}}: 
Let
$$
G^{\textnormal{ran}}_\lambda = \{A_\lambda \in GL(n,\bbR): X(\lambda^E t) \simeq A_\lambda X(t)\}, \quad \lambda > 0,
$$
$$
G^{\textnormal{ran}} = \bigcup_{\lambda > 0} G^{\textnormal{ran}}_\lambda.
$$
Since $X$ is $E$-range o.s.s., $G^{\textnormal{ran}}_\lambda \neq \emptyset$, $\lambda > 0$.



By the proof of Theorem 2.1 in Li and Xiao \cite{li:xiao:2011}, p.\ 1190, $G^{\textnormal{ran}}$ is a subgroup of $GL(n,\bbR)$ which is closed in the relative topology. Moreover, by Lemmas 3.3 and 3.5 in Li and Xiao \cite{li:xiao:2011},
\begin{equation}\label{e:zeta_LiXiao_Glambda_are_disjoint}
\xi:G^{\textnormal{ran}} \rightarrow \bbR, \quad \xi(A) = \log(s) \textnormal{ if } A \in G^{\textnormal{ran}}_s.
\end{equation}
 is a well-defined, continuous homomorphism. 

Now define the continuous group mapping ${\mathcal L} :T(G^{\textnormal{ran}}) \rightarrow \bbR$ by the relation
$$
{\mathcal L} (Q) = \log(\xi (\exp(Q))), \quad Q \in T(G^{\textnormal{ran}}).
$$
In view of \eqref{e:zeta_LiXiao_Glambda_are_disjoint}, the mapping ${\mathcal L}$ is well-defined, since $\exp(Q) \in G^{\textnormal{ran}}$ for $Q\in T(G^{\textnormal{ran}})$, see for example Meerschaert and Scheffler \cite{meerschaert:scheffler:2001}, Proposition 2.2.10.c. By the same argument as on p.137 of Meerschaert and Scheffler \cite{meerschaert:scheffler:2001}, the mapping ${\mathcal L} $ is linear; moreover, ${\mathcal L}$ characterizes the tangent space of the symmetry group $G^{\textnormal{ran}}_1$ as $T(G^{\textnormal{ran}}_1) = \{Q \in T(G^{\textnormal{ran}}): {\mathcal L} (Q) = 0\}$. We would like to show that
\begin{equation}\label{e:E(Phi)={L(E)=1}_X_range}
{\mathcal E}^{\textnormal{ran}}_E (X) = \{Q \in T(G^{\textnormal{ran}}): {\mathcal L}(Q) = 1\}.
\end{equation}
For any $H \in {\mathcal E}^{\textnormal{ran}}_{E}(X)$, $X(\lambda^{E}t) \simeq \lambda^{H}X(t)$, $\lambda > 0$.
Therefore, $\lambda^{H} \in G^{\textnormal{ran}}_{\lambda} \subseteq G^{\textnormal{ran}}$, and from the definition \eqref{e:T(G)}, $H \in T(G^{\textnormal{ran}})$. Consequently, $\xi(e^{H \log \lambda}) = \xi(\lambda^{H}) = \lambda$, implying that ${\mathcal L}(H) = \log \xi(\lambda^{H})\Big|_{\lambda = e} = 1$. Now pick $H \in T(G^{\textnormal{ran}})$ such that ${\mathcal L}(H) = 1$. Then, $Hs \in T(G^{\textnormal{ran}})$, $ s \in \bbR$, whence $\exp(Hs)\in G^{\textnormal{ran}}$. Since \eqref{e:zeta_LiXiao_Glambda_are_disjoint} is a continuous homomorphism, the mapping $s \mapsto \xi(\exp(Hs))$ is a continuous additive homomorphism. Therefore, there is some $\beta \in \bbR$ such that $\log \xi(e^{Hs}) = \beta s$. Since $\log \xi(e^{H}) = 1$, then $\beta = 1$. Therefore, $\log \xi(\exp(H \log \lambda)) = \log \lambda$, whence $\lambda^{H} \in G^{\textnormal{ran}}_{\lambda}$, $\lambda > 0$. In other words, $H \in {\mathcal E}^{\textnormal{ran}}_H(X)$. This proves \eqref{e:E(Phi)={L(E)=1}_X_range}.

By the linearity of ${\mathcal L}$, for any $H$ such that ${\mathcal L}(H) = 1$, ${\mathcal L}(Q) = {\mathcal L}(H) + {\mathcal L}(Q-H)$, where ${\mathcal L}(Q-H) = 0$. This yields
$$
\{Q \in T(G^{\textnormal{ran}}): {\mathcal L}(Q)=1\} = H + \{Q \in T(G^{\textnormal{ran}}): {\mathcal L}(Q)=0\},
$$
which establishes the relation (\ref{e:E=H+TG_range}).
%
%
%

To prove the existence of a commuting exponent, let $A \in G^{\textnormal{ran}}_1$, $H \in {\mathcal E}^{\textnormal{ran}}_E(X)$. A simple computation shows that $\lambda^{AHA^{-1}}=A\lambda^HA^{-1}$, and since $A^{-1}X(t) \simeq X(t)$, it follows that $\lambda^{AHA^{-1}}X(t)=A\lambda^H X(t)\simeq AX(\lambda^E t)\simeq X(\lambda^E t)\simeq \lambda^H X(t)$. Then,
\begin{equation}\label{e:ABA^(-1)=EX_range}
AHA^{-1} \in {\mathcal E}^{\textnormal{ran}}_E(X).
\end{equation}
Let
$$
H_0 = \int_{A \in G^{\textnormal{ran}}_1}AHA^{-1}{\mathbf H}(dA),
$$
where ${\mathbf H}$ is the Haar measure on the compact group $G^{\textnormal{ran}}_1$, so that ${\mathbf H}(S\,dA)={\mathbf H}(dA)$ for any $S\in  G^{\textnormal{ran}}_1$. By the established relation (\ref{e:E=H+TG_range}), ${\mathcal E}^{\textnormal{ran}}_E(X)$ is closed and convex. So, from \eqref{e:ABA^(-1)=EX_range}, we conclude that $H_0 \in {\mathcal E}^{\textnormal{ran}}_E(X)$. Moreover, it is easy to check (compare Meerschaert and Scheffler \cite{meerschaert:scheffler:2001}, p.138) that $AH_0A^{-1} = H_0$ for $A \in G^{\textnormal{ran}}_1$, whence \eqref{e:B0_A=AB_0_X_range} follows.
The last statement is akin to Theorem 5.2.14, Meerschaert and Scheffler \cite{meerschaert:scheffler:2001}, and can be proved in the same way. $\Box$\\

{\sc Proof of Theorem \ref{t:domain_Exponents}}: The proof is similar to Meerschaert and Scheffler \cite{meerschaert:scheffler:2001}, pp.137--138. We outline the main steps, and point out some minor differences.

Recall the definitions of $G$ and $G_{\lambda}$ in expressions \eqref{e:G} and \eqref{e:def_Glambda_X}, respectively, and the mapping $\zeta(\cdot)$ from Lemma \ref{l:HM_lemma6.1_analog}.   As in the proof of Theorem \ref{t:E=H+TG_range}, define the continuous group mapping $L :T(G) \rightarrow \bbR$ by the relation
$$
L (B) = \log(\zeta (\exp(-B))), \quad B \in T(G).
$$
By the same argument as on p.137 of Meerschaert and Scheffler \cite{meerschaert:scheffler:2001}, the mapping $L $ is linear; moreover, $L$ characterizes the tangent space of the symmetry group $G^{\textnormal{dom}}_1$ in the sense that $T(G^{\textnormal{dom}}_1) = \{B \in T(G): L (B) = 0\}$. We need to characterize the set of all exponents in terms of the function $L(\cdot)$, namely, we will show that
\begin{equation}\label{e:E(Phi)={L(E)=1}_X}
{\mathcal E}^{\textnormal{dom}}_H(X) = \{B \in T(G): L(B) = 1\}.
\end{equation}
The argument resembles that for establishing \eqref{e:E(Phi)={L(E)=1}_X_range}, but we lay it out for the reader's convenience. For any $B \in {\mathcal E}^{\textnormal{dom}}_H(X)$, $X(\lambda^{B}t) \simeq \lambda^{H}X(t)$, $\lambda > 0$. Therefore, $\lambda^{-B} \in G_{\lambda} \subseteq G$. Consequently, $\zeta (e^{-B \log(\lambda)}) = \zeta (\lambda^{-B})=\lambda$, implying that $L(B) = \log( \zeta(\lambda^{-B}) )\Big|_{\lambda = e} = 1$. Now pick $B \in T(G)$ such that $L(B) = 1$. Then, $-B s \in T(G)$, $s \in \bbR$, whence $\exp(-B s) \in G$, and $\zeta(\exp(-B s))$ is well defined. As in the proof of Theorem \ref{t:existence_of_domain_exponents}, Lemmas \ref{l:zeta_is_homomorph} and \ref{l:zeta_is_cont} imply that the mapping $s \mapsto \log(\zeta(\exp(-B s)))$ is a continuous additive homomorphism; therefore, there exists $\beta \in \bbR$ such that $\log(\zeta(e^{-B s})) = \beta s$. Since $\log(\zeta(e^{-B})) = 1$, then $\beta = 1$. Therefore, $\log(\zeta(\exp(-B \log(\lambda))))= \log(\lambda)$, whence $\lambda^{- B} \in G_{\lambda}$, $\lambda > 0$. In other words, $B \in {\mathcal E}^{\textnormal{dom}}_H(X)$. This proves \eqref{e:E(Phi)={L(E)=1}_X}.

By the linearity of $L$, for any $E$ such
that $L(E) = 1$, $L(B) = L(E) + L(B-E)$, where $L(B-E) = 0$. This yields
$$
\{B \in T(G): L(B)=1\} = E + \{B \in T(G): L(B)=0\},
$$
which establishes the relation (\ref{e:E(varphi)=B+T(S(varphi))}).

We now prove the existence of a commuting exponent. Notice that for any $A \in G^{\textnormal{dom}}_1$, $B \in {\mathcal E}^{\textnormal{dom}}_H(X)$,
$X(\lambda^{AB A^{-1}}t) \simeq X(\lambda^{B}t) \simeq \lambda^{H}X(t)$, $\lambda > 0$, so that
\begin{equation}\label{e:ABA^(-1)=E(varphi)_X}
AB A^{-1} \in {\mathcal E}^{\textnormal{dom}}_H(X).
\end{equation}
Let
$$
B_0 = \int_{A \in G^{\textnormal{dom}}_1}AB A^{-1}{\mathbf H}(dA),
$$
where ${\mathbf H}$ is the Haar measure on the compact group $G^{\textnormal{dom}}_1$. By the relation (\ref{e:E(varphi)=B+T(S(varphi))}), ${\mathcal E}^{\textnormal{dom}}_H(X)$ is closed and convex. So, from \eqref{e:ABA^(-1)=E(varphi)_X}, we conclude that $B_0 \in {\mathcal E}^{\textnormal{dom}}_H(X)$. Moreover, the same argument as in Meerschaert and Scheffler \cite{meerschaert:scheffler:2001}, p.138, yields $AB_0A^{-1} = B_0$, from which \eqref{e:B0_A=AB_0_X} follows.

The last statement is akin to Theorem 5.2.14, Meerschaert and Scheffler \cite{meerschaert:scheffler:2001}, p.\ 139, and can be proved in the same way. $\Box$\\

{\sc Proof of Corollary \ref{c:relation_between_domain_and_range_exponents}}: Equation \eqref{e:relation_between_range_exponents} follows easily from \eqref{e:E=H+TG_range}, and equation \eqref{e:relation_between_domain_exponents} is a direct result of \eqref{e:E(varphi)=B+T(S(varphi))}.  $\Box$\\

Finally we come to the proof of Theorem \ref{t:X=X1+X2}, where we relax the assumption that every eigenvalue of $H$ has positive real part. For this purpose, in the sequel we will state and prove Proposition \ref{p:minRe(E)=0<=>minRe(H)=0} and Lemmas \ref{l:invariant_law_Re=0}--\ref{l:HM_lemma6.1_analog}.

\begin{prop}\label{p:minRe(E)=0<=>minRe(H)=0}
Suppose $X$ is a proper, stochastically continouous random vector field that satisfies the scaling relation \eqref{e:o.s.s.}.  Then:
\begin{itemize}
\item [(i)] There is no pair of eigenvalues $e$ and $h$ for $E$ and $H$, respectively, whose real parts have opposite signs;
\item [(ii)] If every eigenvalue of $H$ has positive real part, then every eigenvalue of $E$ has positive real part;
\item [(iii)] If $X(0)=0$ a.s., and if every eigenvalue of $E$ has positive real part, then every eigenvalue of $H$ has positive real part.
\end{itemize}
\end{prop}

\begin{proof}
($i$) Without loss of generality, assume by contradiction that there are eigenvalue $e$ and $h$ of $E$ and $H$, respectively, such that $\Re(e) > 0$ and $\Re(h) < 0$; otherwise, we can pick the pair of exponents $(-E,-H)$, instead. Let $\{c_k\}_{k \in \bbN}$ be a sequence of positive numbers such that $c_k \rightarrow \infty$. Then, $\|c^E_k\| \rightarrow \infty$, since the eigenvalue $c_k^e$ of $c^{E}_k$ goes to infinity in ${\mathbb C}$. By Lemma \ref{e:A^(-1)w_k->0}, there is a subsequence $\{t_{k'}\} \subseteq S^{m-1}_{\bbR}$ such that $c^{-E}_{k'}t_{k'} \rightarrow 0$. Choose a further subsequence $\{t_{k''}\}$ such that $ t_{k''} \rightarrow t_0$, for some $t_0 \in S^{m-1}_{\bbR}$. For notational simplicity, we drop the superscript and write $k$. By operator self-similarity, $c^{H}_k X(c^{-E}_k t_k) \stackrel{d}= X(t_k)$. The Jordan form $H = PJ_H P^{-1}$ yields
$$
c^{J_H}_k P^{-1}X(c^{-E}_k t_k) \stackrel{d}= P^{-1}X(t_k).
$$
Let $Y(c^{-E}_k t_k) = P^{-1}X(c^{-E}_k t_k) \in \bbC^n$. There is a $j \times j$ Jordan block $J_h$ in $J_H$ associated with the eigenvalue $h$; for simplicity, we can assume that $J_h$ occupies the upper left $j \times j$ block in $J_H$. Let $\pi_{\leq j}$ be the projection operator onto the first $j$ entries of a vector in $\bbC^n$. By the continuity in probability of the random field $X$, $Y(c^{-E}_k t_k) \stackrel{P}\rightarrow  P^{-1}X(0)$, $k \rightarrow \infty$. Since $\Re(h) < 0$,
$$
0 \stackrel{P}\leftarrow \pi_{\leq j}[c^{J_H}_k Y(c^{-E}_k t_k)] \stackrel{d}= \pi_{ \leq j}[P^{-1}X(t_k)] \stackrel{P}\rightarrow \pi_{\leq j}[P^{-1}X(t_0)] ,
$$
which contradicts the properness of $X(t_0)$. \\

($ii$) Suppose that $e=ib$ is an eigenvalue of $E$ with zero real part. The Jordan form of the matrix exponential $c^{E} = Pc^{J_E}P^{-1}$, $P \in GL(n,\bbC)$, reveals that $c^E$ cannot converge to 0 as $c \rightarrow 0^+$, since its eigenvalue $c^{e} = c^{i b}$ remains bounded from below (and above). Therefore, there exist $c_0,m>0$ and $t_0\neq 0$ such that $\|c^{E}t_0\| > m$ for all $0<c<c_0$.  Since $\{c^{E}t_0:0<c<c_0\}$ is relatively compact, there exists a sequence $c_k\to 0$ such that $c^{E}_k t_0\to t_1\neq 0$, and then $X(c_k^{E}t_0)\to X(t_1)$ in distribution, where $X(t_1)$ is full. If every eigenvalue of $H$ has positive real part, then $\|c^H\|\to 0$ as $c\to 0$, and hence $c^{H} X(t_0) \to 0$ in probability, which is a contradiction.\\

($iii$) If every eigenvalue of $E$ has positive real part, then $\|c^E\|\to 0$ as $c\to 0$, and hence $X(c^E t) \to 0$ in probability as $c\to 0$ for any $t\in\bbR^m$.  Suppose that $h=ia$ is an eigenvalue of $H$ with zero real part, and hence also an eigenvalue of the transpose $H^*$, the linear operator such that the inner product relation $\langle H x,y\rangle=\langle x,H^*y\rangle$ holds for all $x,y\in\bbR^n$.  As in ($ii$), it follows that there exists a vector $x_0$ and a sequence $c_k\to 0$ such that $c_k^{H^*}x_0\to x_1\neq 0$.  Then $\langle x_0,c_k^{H}X(t)\rangle=\langle c_k^{H^*}x_0,X(t)\rangle\to\langle x_1,X(t)\rangle$ in distribution, and since $X(t)$ is full, we arrive at a contradiction. $\Box$\\
\end{proof}

For the next lemma, recall that $O(n)$ denotes the orthogonal group in $GL(n,\bbR)$.
\begin{lem}\label{l:invariant_law_Re=0}
Let $H \in M(n,\bbR)$ be a diagonalizable matrix (over ${\mathbb C}$) whose eigenvalues have zero real parts. Then, there exists a Gaussian random vector $X$ such that
\begin{equation}\label{e:r^H_X_d=_X}
r^{H}X \stackrel{d}= X, \quad r > 0.
\end{equation}
\end{lem}
\begin{proof}
The proof is by construction. By the Jordan decomposition of $H$ over the field $\bbR$ (see Meerchaert and Scheffler \cite{meerschaert:scheffler:2001}, Theorem 2.1.16), there exists a conjugacy $P \in GL(n,\bbR)$ such that $H = P J_{H,\bbR}P^{-1}$, where $J_{H,\bbR} = \textnormal{diag}(J_1,\hdots,J_q)$.
Each block $J_j$, $j = 1,\hdots,q$, is either the scalar zero or has the form
$$
J_j
= \left(\begin{array}{cc}
0 & - \theta_j\\
\theta_j & 0
\end{array}\right) .
$$
Therefore, $\exp\{ c \hspace{0.5mm}\textnormal{diag}(J_1,\hdots,J_q) \} \in O(n)$ for any $c \in \bbR$. In particular, this holds for $c = 1$. Now let $X = PZ$, where $Z \sim N(0,I)$. Then, \eqref{e:r^H_X_d=_X} holds, since
\begin{equation}\label{e:OZ=d_Z}
O Z \stackrel{d}= Z \textnormal{ for any $O \in O(n)$}. \quad \Box
\end{equation}
\end{proof}
\begin{remark}
In Lemma \ref{l:invariant_law_Re=0}, the Gaussian distribution is not essential. The argument holds with any random vector $Z$ displaying a spherical distribution, namely, one that satisfies \eqref{e:OZ=d_Z}. For example, for $n = 2$, $Z$ can have density $f_Z(z) = C (1 + \|z\|^{\beta})^{-1}$ for a normalizing constant $C > 0$ and some $\beta > 2$, where $\|\cdot\|$ denotes the Euclidean norm.
\end{remark}

\begin{lem}\label{l:HM_lemma6.3_analog}
Assume that $X$ is a proper, stochastically continuous,  random vector field that satisfies the scaling relation \eqref{e:o.s.s.} for some $E$ whose eigenvalues all have positive real part.   Let $f$ be the minimal polynomial of $H$, and write
\begin{equation}\label{e:f=gh}
f = f_1f_2,
\end{equation}
where the roots of $f_1$ have zero real part, and the roots of $f_2$ have positive real part. Write the direct sum decomposition $\bbR^n = V_1 \oplus V_2$ where $V_1 = \textnormal{Ker} \hspace{1mm}f_1(H)$, $V_2 = \textnormal{Ker} \hspace{1mm}f_2(H)$. Write  $X=X_1+X_2$ and $H=H_1\oplus H_2$ with respect to this direct sum decomposition.  Then, $X_2$ is a proper ($E,H_2$)--o.s.s.\ random field on $V_2$.
\end{lem}

\begin{proof}
Let $\pi_2:\bbR^n\to V_2$ denote the projection operator defined by $\pi_2(v)=v_2$, where $v = v_1 + v_2$, for some unique $v_1 \in V_1$ and $v_2 \in V_2$ by the direct sum decomposition. Then $H v =  Hv_1 + H v_2=H_1 v + H_2 v$ . Hence $\pi_2(Hv) = H v_2= H_2 v$, which leads to the commutativity relation $\pi_{2}c^{H} = c^{H_2}\pi_{2}$.  This in turn implies that
$$
\{\pi_2 X(c^E t)\}_{t \in \bbR^m} \simeq \{c^{H_2}\pi_2 X(t)\}_{t \in \bbR^m}.
$$
Therefore, $X_2=\pi_{2}X$ is a proper, stochastically continuous ($E,H_2$)--o.s.s.\ random field on $V_2$. $\Box$\\
\end{proof}

The next lemma uses non-Euclidean polar coordinates as in Jurek and Mason \cite{Jurek1993}, Proposition 3.4.3, see also Meerschaert and Scheffler \cite{meerschaert:scheffler:2001} and Bierm\'{e} et al.\ \cite{bierme:meerschaert:scheffler:2007}. Suppose the real parts of the eigenvalues of $E \in M(m,\bbR)$ are positive. Then, there exists a norm $\| \cdot \|_0$ on $\bbR^m$ for which
\begin{equation}\label{e:Phi_change-of-variables}
\Psi:(0,\infty) \times S_0 \rightarrow \bbR^m \backslash\{0\},
\quad \Psi(r,\theta):=r^{E}\theta,
\end{equation}
is a homeomorphism, where $S_0 = \{x \in \bbR^m: \norm{x}_{0}=1\}$. One can then uniquely write the polar coordinates representation
\begin{equation}\label{e:x=tau(x)E*l(x)}
\bbR^m \backslash\{0\}\ni x = \tau_E(x)^E l_E(x),
\end{equation} where $\tau_E(x) > 0$, $l_E(x) \in S_0$ are called the radial and directional parts, respectively. One such norm $\| \cdot \|_0$ may be calculated explicitly by means of the expression
\begin{equation}\label{e:E0_norm}
\norm{x}_{0} = \int^{1}_{0} \norm{t^{E}x}_{*}\frac{dt}{t},
\end{equation}
where $\| \cdot \|_{*}$ is any norm in $\bbR^m$.
The uniqueness of the representation (\ref{e:x=tau(x)E*l(x)}) yields
\begin{equation}\label{e:tau_E}
\tau_E(c^E x) = c \tau_E(x), \quad l_E(c^E x) = l_E(x).
\end{equation}

\begin{lem}\label{l:admissibility_of_H}
Suppose that every eigenvalue of $E$ has positive real part.  Then the following are equivalent:
\begin{itemize}
\item[(i)] $H \in {\mathcal E}^{\textnormal{ran}}_E(X)$ for some $E$-range operator self-similar random field $X$;
\item[(ii)] every eigenvalue of $H$ has nonnegative real part, and every eigenvalue with null real part is a simple root of the minimal polynomial of $H$.
\end{itemize}
\end{lem}

\begin{proof} The proof is a direct extension of the argument in Hudson and Mason \cite{hudson:mason:1982}, Theorem 3.
First we show that ($i$) implies ($ii$). Since we assume that every eigenvalue of $E$ has positive real part, then by Proposition \ref{p:minRe(E)=0<=>minRe(H)=0}, ($i$), it follows that every eigenvalue of $H$ has nonnegative real part. Then we just need to show that every eigenvalue of $H$ having null real part is a simple root of the minimal polynomial of $H$. For $m =1$, our proof is mathematically equivalent to Theorem 3 in Hudson and Mason \cite{hudson:mason:1982}, although we substantially simplify the technical phrasing of the argument.

Suppose by contradiction that some eigenvalue $h=ib$ of $H$ is not a simple root of the minimal polynomial of $H = PJ_HP^{-1}$. Then, in the Jordan decomposition of $H$ there is a non-diagonal $j \times j$ Jordan block $J_{H,h}$ associated with $h$. We can assume that $J_{H,h}$ corresponds to the upper left $j \times j$ block. Let $\{t_k\}_{k \in \bbN} \subseteq \bbR^m$ be a sequence such that $t_k \rightarrow 0$. Consider the polar decomposition $t_k = \tau_E(t_k)^{E}l_E(t_k)$. By the compactness of $S_0$, we can assume, without loss of generality, that $l_E(t_k) \rightarrow l_0 \in S_0$ as $k \rightarrow \infty$. By operator self-similarity, $(\tau_E(t_k)^{-1})^{H} X(t_k) \stackrel{d}= X(l_E(t_k))$. Therefore,
$$
(\tau_E(t_k)^{-1})^{J_H} P^{-1} X(t_k) \stackrel{d}= P^{-1} X(l_E(t_k)).
$$
Let $Y(t_k) = P^{-1} X(t_k) \in \bbC^n$, $Y(l_E(t_k))= P^{-1} X(l_E(t_k)) \in \bbC^n$, and let $\pi_{\leq j}$ be the projection operator on the first $j$ entries of a vector in $\bbC^n$. Then,
\begin{equation}\label{e:scaling_law_under_projection}
\pi_{\leq j} [ (\tau_E(t_k)^{-1})^{J_H} Y(t_k) ] \stackrel{d}= \pi_{\leq j} [ Y(l_E(t_k)) ] ,
\end{equation}
where, by continuity in probability, $\pi_{\leq j} [ Y(l_E(t_k)) ] \stackrel{P}\rightarrow \pi_{\leq j} [Y(l_0) ]$ as $k \rightarrow \infty$. Moreover, by the expression for the matrix exponential (see, for instance, Didier and Pipiras \cite{didier:pipiras:2011}, p.\ 31, expression (D2)),
$$
\pi_{\leq j} [ (\tau_E(t_k)^{-1})^{J_H} Y(t_k) ]
= \left(\begin{array}{ccccc}
1 & 0 & 0 & \hdots & 0 \\
\log \tau_E(t_k)^{-1} & 1 & 0 & \hdots & 0\\
\frac{\log^2 \tau_E(t_k)^{-1}}{2!} & \log \tau_E(t_k)^{-1} & 1 & \hdots & 0\\
\vdots & \ddots & \ddots & \hdots & 0\\
\frac{\log^{j-1} \tau_E(t_k)^{-1}}{(j-1)!} & \frac{\log^{j-2} \tau_E(t_k)^{-1}}{(j-2)!} & \hdots & \log \tau_E(t_k)^{-1}  & 1\\
\end{array}\right)
\left(\begin{array}{c}
Y_1(t_k)\\
Y_2(t_k)\\
\vdots \\
Y_{j-1}(t_k)\\
\end{array}\right).
$$
Looking at the first two entries of \eqref{e:scaling_law_under_projection}, we arrive at the system
\begin{equation}\label{e:bivariate_with_logs_contra}
\left(\begin{array}{c}
Y_{1}(t_k)\\
-\log \tau_E(t_k) \hspace{1mm}Y_{1}(t_k) + Y_{2}(t_k)
\end{array}\right)
\stackrel{d}= \left(\begin{array}{c}
Y_{1}(l_E(t_k))\\
Y_{2}(l_E(t_k))
\end{array}\right).
\end{equation}
Since the term $-\log \tau_E(t_k)\to\infty$ as $k\to\infty$ and $Y(l_E(t_k)) \stackrel{P}\rightarrow Y(l_0)$, we have $Y_{1}(t_k) \stackrel{P}\rightarrow 0$. In view of the first entry of the relation \eqref{e:bivariate_with_logs_contra}, this contradicts the properness of $Y(l_0)$.

The proof that ($ii$) implies ($i$) is by construction. Write the direct sum decomposition $\bbR^n = V_1 \oplus V_2$ as in Lemma \ref{l:HM_lemma6.3_analog}. Write $H=H_1\oplus H_2$ with respect to this direct sum decomposition, so that $H_1$ is semisimple (diagonalizable over ${\mathbb C}$). Since every eigenvalue of $H_1$ has zero real part, the closure of the family $\overline{\{r^{H_1}: r > 0\}}$ in the operator topology is a compact group of linear operators on $V_1$.

Now let $X_1$, $X_2$ be two independent random vectors which are full and take values in $V_1$ and $V_2$, respectively. By Lemma \ref{l:invariant_law_Re=0}, we can further assume that the distribution of $X_1$ is invariant under the group
\begin{equation}\label{e:group_r^H1:r>0_closure}
\overline{\{r^{H_1}: r > 0\}}.
\end{equation}
Since the eigenvalues of $E$ have positive real parts, we can define the random field $X = \{X(t)\}_{t \in \bbR^m}$ by
$$
X(t) = X( \tau_E(t)^E l_E(t)) := \tau_E(t)^H (X_1 + X_2), \quad t \in \bbR^m \backslash\{0\}, \quad X(1) = X_1.
$$
In particular, $X( \theta ) = X_1 + X_2$, $\theta \in S_0$, and
\begin{equation}\label{e:scaling_in_V1_V2}
X(t) = \tau_E(t)^{H_1} X_1 + \tau_E(t)^{H_2}X_2 \stackrel{d}= X_1 + \tau_E(t)^{H_2}X_2, \quad t \neq 0,
\end{equation}
where the second equality in law follows from the invariance of the distribution of $X_1$ under the group \eqref{e:group_r^H1:r>0_closure} and the independence between $X_1$ and $X_2$. The random field $X$ is proper, satisfies the scaling relation $X(c^Et) \simeq c^H X(t)$, $c > 0$, and is continuous in probability at every $t \in \bbR^m \backslash\{0\}$. Now take a sequence $\{t_k\}_{k \in \bbN} \subseteq \bbR^m \backslash\{0\}$, $t_k \rightarrow 0$, $k \in \bbN$. Then, by \eqref{e:scaling_in_V1_V2} at $t_k$ and the fact that the eigenvalues of $H_2$ have positive real parts, $X(t_k) \stackrel{d}\rightarrow X_1$ as $k \rightarrow \infty$, i.e., $X$ is continuous in law at every $t \in \bbR^m$. $\Box$
\end{proof}

\begin{lem}\label{l:HM_lemma6.2_analog}
Assume that $X$ is a proper, stochastically continuous random vector field that satisfies the scaling relation \eqref{e:o.s.s.} for some $E$ whose eigenvalues all have positive real part.
Then, every eigenvalue of $H$ has real part equal to zero if and only if $X(0)$ is full.
\end{lem}

\begin{proof}
Assume that the distribution of $X(0)$ is full, and suppose by contradiction that some eigenvalue of $H$, and thus of $H^*$, has real part different from zero. Then, by Lemma \ref{l:admissibility_of_H}--(ii), such an eigenvalue has positive real part. As a consequence, there is $v \in \bbR^n \backslash \{0\}$ such that $\lim_{c \rightarrow 0^+}c^{H^*}v = 0$. Therefore, for $t \neq 0$, and by the assumption that $\min \Re (\textnormal{eig}(E)) > 0$,
$$
v^* X(0)  \stackrel{d}\leftarrow v^* X(c^E t) \stackrel{d}= v^* c^{H}X(t) \rightarrow 0^*X(t) = 0, \quad c \rightarrow 0^+.
$$
This contradicts the properness of $X(0)$.

Conversely, assume every eigenvalue of $H$ has real part equal to zero, and suppose by contradiction that the distribution of $X(0)$ is not full. Then, there is $v \neq 0$ such that $v^*X(0) = 0$. But by Lemma \ref{l:admissibility_of_H} the eigenvalues of $H$ are simple roots of the minimal polynomial of $H$. Therefore, $\overline{\{c^{H}: c > 0\}}$ has a compact closure in $GL(n,\bbR)$, whence one can pick a sequence $\{c_k\}$ such that $c_k \rightarrow 0^+$ and $c^{H}_k \rightarrow A \in GL(n,\bbR)$. Thus, since every eiganvalue of $E$ has positive real part,
$$
X(0) \stackrel{P}\leftarrow X(c^E_{k}t) \stackrel{d}= c^{H}_k X(t) \stackrel{d}\rightarrow A X(t).
$$
We arrive at $0 = v^* X(0) \stackrel{d}= v^* AX(t)$, which contradicts the properness of $X$. $\Box$\\
\end{proof}

\begin{lem}\label{l:HM_lemma6.1_analog}
Assume that $X$ is a proper, stochastically continuous random vector field that satisfies the scaling relation \eqref{e:o.s.s.} for some $E$ whose eigenvalues all have positive real part.  If $X(0)$ is full, there is a version of $X$ with constant sample paths.
\end{lem}

\begin{proof}
Under the assumptions, by Lemma \ref{l:HM_lemma6.2_analog}, every eigenvalue of $H$ has zero real part. So, by Lemma \ref{l:admissibility_of_H}, the eigenvalues of $H$ are simple roots of the minimal polynomial of $H$. Consequently, $H$ is diagonalizable over ${\mathbb C}$ with all roots having zero real parts. The group $\overline{\{c^H: c > 0\}}$, where the closure is taken in $GL(n,\bbR)$, is then compact. We can pick a sequence $\{c_{k}\}_{k \in \bbN}$, such that $c_k \rightarrow 0^+$ and $c^{H}_k \rightarrow A$ for some $A \in GL(n,\bbR)$. Then, for an arbitrary $q$-tuple $t_1,\hdots,t_q \in \bbR^m$, $q \in \bbN$,
$$
(X(0),\hdots,X(0)) \stackrel{P}\leftarrow (X(c^{E}_k t_1),\hdots,X(c^{E}_k t_q))
 $$
 $$
 \stackrel{d}= (c^{H}_{k}X(t_1),\hdots, c^{H}_{k}X(t_q)) \stackrel{P}\rightarrow (AX(t_1),\hdots, AX(t_q)).
$$
In particular, $A^{-1}X(0) \stackrel{d}= X(t) \stackrel{P}\rightarrow X(0)$, as $t \rightarrow 0$. Thus, $\{X(t)\}_{t \in \bbR^m} \simeq \{X(0)\}_{t \in \bbR^m}$. Now let $Z(t) = X(0)$, $t \in \bbR^m$. Then,
\begin{equation}\label{e:Z=X}
Z = \{Z(t)\}_{t \in \bbR^m} \simeq \{X(t)\}_{t \in \bbR^m}
\end{equation}
 and $Z$ has constant sample paths. Consider $\bbQ^n$ and define the set of functions (sample paths)
$$
{\mathcal D} = \bigcap_{s \in \bbQ^n} \{f : \bbR^m \rightarrow \bbR^n: f(s) = f(0)\}.
$$
Then, $P(\{X(t)\} \in {\mathcal D}) = P(\{Z(t)\} \in {\mathcal D}) = 1$, by \eqref{e:Z=X}. In particular, for $t_0 \in \bbQ^n$, $P(X(t_0) = X(0) = Z(t_0)) = 1$. For $t'_0 \notin \bbQ^n$, consider a sequence $\{t_k\} \subseteq \bbQ^n$ such that $t_k \rightarrow t'_0$. Then, $Z(t'_0)  = X(0) = X(t_k) \rightarrow X(t'_0)$ in probability, where the equalities hold a.s.\ and the limit is a consequence of continuity in probability. Therefore, $Z(t'_0) = X(t'_0)$ a.s., i.e., $P(Z(t) = X(t)) = 1$, $t \in \bbR^m$, as claimed. $\Box$\\
\end{proof}

{\sc Proof of Theorem \ref{t:X=X1+X2}}: The proof is akin to Theorem 4 in Hudson and Mason \cite{hudson:mason:1982}. We provide the details for the reader's convenience.
Recall the decomposition \eqref{e:f=gh} of the minimal polynomial of $H$, where the roots of $f_1$ have zero real parts. Let $\pi_2$ be the projection operator onto $V_2$ defined by the direct sum decomposition. By Lemma \ref{l:HM_lemma6.3_analog}, the restriction $\{\pi_2 X(t)\}$ is ($E,H_2$)--o.s.s.\ on $V_2$. Since every eigenvalue of $H_1$ has real part zero, it follows from Lemma \ref{l:HM_lemma6.2_analog} that $\pi_1 X(0)$ is full in $V_1$. Hence, by Lemma \ref{l:HM_lemma6.1_analog}, there is a version $\{X_1(t)\}$ of $\{\pi_1 X(t)\}$ with constant sample paths. Moreover, every eigenvalue of $E$ has positive real part, and every eigenvalue of $H_2$ has positive real part,
$$
\pi_2 X(0) \stackrel{P}\leftarrow \pi_2 X (c^{E}t) \stackrel{d}= \pi_2 c^{H}X(t) \stackrel{d}\rightarrow 0, \quad c \rightarrow 0^+.
$$
This establishes ($i$) and ($ii$). $\Box$

\bibliography{ofbf}

\small

\bigskip

\noindent \begin{tabular}{lll}
Gustavo Didier & Mark M.\ Meerschaert & Vladas Pipiras \\
Mathematics Department &  Dept.\ of Statistics and Probability & Dept.\ of Statistics and Operations Research \\
Tulane University & Michigan State University & UNC at Chapel Hill \\
6823 St.\ Charles Avenue  & 619 Red Cedar Road & CB\#3260, Hanes Hall \\
New Orleans, LA 70118, USA & East Lansing, MI 48824, USA & Chapel Hill, NC 27599, USA \\
{\it gdidier@tulane.edu}& {\it mcubed@stt.msu.edu}   & {\it pipiras@email.unc.edu} \\
\end{tabular}\\

\smallskip

\end{document}